\documentclass[11pt, a4paper]{article}
\usepackage{amssymb,amsmath,amsthm}
\usepackage{fancyhdr}
\usepackage[british]{babel}
\usepackage{enumitem}
\usepackage{appendix}
\usepackage{graphicx}

\usepackage{hyperref}
\usepackage[margin=2.5cm]{geometry}

\newtheorem{theorem}{Theorem}[section]
\newtheorem{prop}[theorem]{Proposition}
\newtheorem{lemma}[theorem]{Lemma}
\newtheorem{cor}[theorem]{Corollary}
\newtheorem{conjecture}[theorem]{Conjecture}

\theoremstyle{definition}

\newtheorem{defn}[theorem]{Definition}

\newtheorem{claim}[theorem]{Claim}

\newcommand*{\abs}[1]{\lvert #1\rvert}

\newcommand{\ep}{\varepsilon}

\title{Balanced subdivisions of cliques in graphs}

\author{
	Bingyu Luan\thanks{School of Mathematics, Shandong University, China. Email: {\tt  byluan@mail.sdu.edu.cn, ghwang@sdu.edu.cn}. B.L. and G.W. were supported by Natural Science Foundation of China (11871311), Young Taishan Scholars Program
	and seed fund program for international research cooperation of Shandong University.}
	\and
	Yantao Tang\thanks{Zhongtai Securities Institute for Financial Studies, Shandong University, China. Email: {\tt yttang@mail.sdu.edu.cn}.}
	\and
	Guanghui Wang\footnotemark[1]
	\and
	Donglei Yang\thanks{Data Science Institute, Shandong University, China. Email: {\tt dlyang@sdu.edu.cn}. Supported by the China Postdoctoral Science Foundation (2021T140413), Natural Science Foundation of China (12101365) and Natural Science Foundation of Shandong Province (ZR2021QA029).}
}
\date{}
\begin{document}
\maketitle
\begin{abstract}
   Given a graph $H$, a balanced subdivision of $H$ is a graph obtained from $H$ by subdividing every edge the same number of times. In 1984, Thomassen conjectured that for each integer $k\ge 1$, high average degree is sufficient to guarantee a balanced subdivision of $K_k$. Recently, Liu and Montgomery resolved this conjecture. We give an optimal estimate up to an absolute constant factor by showing that there exists $c>0$ such that for sufficiently large $d$, every graph with average degree at least $d$ contains a balanced subdivision of a clique with at least $cd^{1/2}$ vertices. It also confirms a conjecture from Verstra{\"e}te: every graph of average degree $cd^2$, for some absolute constant $c>0$, contains a pair of disjoint isomorphic subdivisions of the complete graph $K_d$. We also prove that there exists some absolute $c>0$ such that for sufficiently large $d$,  every $C_4$-free graph with average degree at least $d$ contains a balanced subdivision of the complete graph $K_{cd}$, which extends a result of Balogh, Liu and Sharifzadeh.
\end{abstract}

\section{Introduction}

Given a graph $H$, a \emph{subdivision} of $H$, denote by $\mathsf{T}H$, is a graph obtained from $H$ by subdividing some or all of its edges by drawing new vertices on those edges. In other words, some edges of $H$ are replaced by internally vertex-disjoint paths. The original vertices of $H$ are the \emph{branch vertices} of the $\mathsf{T}H$, and its new vertices are called \emph{subdividing vertices}. Subdivision  plays a central role in topological graph theory since Kuratowski \cite{Kuratowski} showed that a graph is planar if and only if it does not contain a subdivision of a complete graph on five vertices or a subdivision of a complete bipartite graph with three vertices in each partition.

In 1967, Mader \cite{Mader-1} proved that for each integer $k\ge 1$, there exists $c>0$ such that every graph with average degree at least $c$ contains a subdivision of the complete graph $K_k$. For each integer $k\ge 1$, let $d(k)$ be the smallest number such that each graph with average degree at least $d(k)$ contains a subdivision of $K_k$. Mader \cite{Mader-1}, and independently Erd\H{o}s and Hajnal \cite{E-H} conjectured that $d(k)=O(k^2)$. The extremal example, the disjoint union of complete regular bipartite subgraphs, first observed by Jung \cite{Jung}, gives a lower bound of $d(k)$, which matches the conjecture. Later, Bollob{\'a}s and Thomason \cite{B-Th}, independently Koml{\'o}s and Szemer{\'e}di \cite{K-Sz-1,K-Sz-2} confirmed the conjecture in 1990s.

The extremal example suggests that this bound can be improved when some small subgraphs are forbidden. Mader \cite{Mader-3} conjectured that every $C_4$-free graph contains a subdivision of a clique with order linear in its average degree. K{\"u}hn and Osthus \cite{K-O-1,K-O-3} proved that every graph with sufficiently large girth contains a subdivision of a clique with order linear in its minimum degree. They \cite{K-O-2} also showed that each $C_4$-free graph $G$ with average degree $k$ contains a $\mathsf{T}K_{k/\log^{12}k}$. Balogh, Liu and Sharifzadeh \cite{BLS} proved when $k\ge 3$, each $C_{2k}$-free graph contains a subdivision of a clique with order linear in its average degree. In 2016, Liu and Montgomery \cite{LM1} solved this conjecture.

A natural question of subdivision is whether we can control the length of paths which replace the edges. Given an integer $z$, denoted by $\mathsf{T}H^{(z)}$ the subdivision obtained from $H$ by replacing all edges of $H$ with internally vertex-disjoint paths of length $z$, and we also call it a \emph{balanced} $H$-subdivision.
An old question of Erd\H{o}s \cite{E} is, for each $\ep>0$, whether one can find a $\delta>0$ such that every graph with $n$ vertices and at least $\ep n^2$ edges contains a $\mathsf{T}K_{\delta \sqrt{n}}^{(2)}$. Alon, Krivelevich and Sudakov \cite{AKS} showed the existence of $\delta$, and this result was improved to $\delta=\ep$ by Fox and Sudakov \cite{F-S}. Erd\H{o}s' question only considers dense graphs. Generally, Thomassen \cite{Th-1,Th-2,Th-3} gave the following conjecture.

\begin{conjecture}[\cite{Th-3}]\label{conj}
	For every $k\ge 1$, there exists a function $f(k)$, if $\delta(G)\ge f(k)$, then $G$ contains a balanced subdivision of $K_k$.
\end{conjecture}

Recently, Liu and Montgomery \cite{LM2} confirmed Conjecture \ref{conj}. Later, a result of Wang \cite{WY} implies balanced clique subdivision of order $d^c$ for any $c<1/2$. We give the following theorem, improving Wang's result.

\begin{theorem}\label{thm: balanced subdivision}
	There exists an absolute constant $c>0$ such that every graph with average degree at least $d$ contains a $\mathsf{T}K_{c\sqrt{d}}^{(z)}$ for sufficiently large $d$ and some $z\in \mathbb{N}$.
\end{theorem}

Note that the bound above is asymptotically optimal up to a constant factor by considering a disjoint union of complete bipartite graphs $K_{d,d}$. This theorem also resolves the following conjecture from Verstra{\"e}te \cite{V}: every graph of average degree $cd^2$, for some absolute constant $c>0$, contains a pair of disjoint isomorphic subdivisions of the complete graph $K_d$. Indeed, we can divide a balanced subdivision of $K_{d}$, which can be found by Theorem \ref{thm: balanced subdivision}, into two balanced subdivisions of $K_{d/2}$, which are isomorphic.

Mader's conjecture states that every $C_4$-free graph has a subdivision of a clique with order linear in its average degree. Balogh, Liu and Sharifzadeh \cite{BLS} proved that every $n$-vertex $C_4$-free graph with $cn^{3/2}$ edges contains a $\mathsf{T}K_{c'\sqrt{n}}^{(4)}$ for some $c'>0$. We extend their result by giving the following theorem. The bound is also optimal up to a constant factor.

\begin{theorem}\label{thm: c4-free balanced subdivision}
	There exists an absolute constant $c>0$ such that every $C_4$-free graph with average degree at least d contains a $\mathsf{T}K_{cd}^{(z)}$ for sufficiently large $d$ and some $z\in \mathbb{N}$.
\end{theorem}

We shall give a unified approach for Theorem \ref{thm: balanced subdivision} and Theorem \ref{thm: c4-free balanced subdivision}. The rest of the paper will be organized as follows. In Section \ref{sec:prelim}, we introduce some necessary notions and the main result (Lemma \ref{thm: dense}) in our proofs. Section \ref{sec:proof1} is devoted to the proof of Lemma \ref{thm: dense} and in Section \ref{sec:dense proof} we discuss three crucial ingredients used for embedding balanced subdivisions.

\section{Preliminaries}\label{sec:prelim}

\subsection{Notation}

Given a graph $G=(V,E)$, we write $\abs{G}=\abs{V(G)}$ for the order of the graph $G$. Let $\delta(G)$ and $d(G)$ be the minimum and average degree of $G$ respectively. For a set of vertices $W\subseteq V(G)$, denote its external neighbourhood by $N(W)=(\cup_{v\in W}N(v))\backslash W$. Denote by $G[W]$ the induced subgraph of $G$ on $W$, and we write $G-W$ for the induced subgraph $G[V(G)\backslash W]$. Given graphs $G$ and $H$, the graph $G\cup H$ has vertex set $V(G)\cup V(H)$ and edge set $E(G)\cup E(H)$. For a collection $\mathcal{P}$ of graphs, denote by $\abs{\mathcal{P}}$ the number of graphs in $\mathcal{P}$, and we write $V(\mathcal{P})=\cup_{G\in \mathcal{P}}V(G)$.

For a path $P$, the \emph{length} of $P$, denoted by $l(P)$, is the number of edges in $P$. Given two vertices $x$, $y$, an $x,y$-path is a path with endvertices $x$ and $y$. When we say $P$ is a path from a vertex set $A$ to a disjoint vertex set $B$, we mean that $P$ has one endvertex in $A$ and another one in $B$, and has no internal vertices in $A\cup B$.

Let $[n]:=\{1,2,\cdots,n\}$. When it is not essential, we omit the floors and ceilings. All logarithms are natural.

\subsection{Koml\'os-Szemer\'edi graph expander}

Koml\'os and Szemer\'edi \cite{K-Sz-1,K-Sz-2} introduced the following sublinear expander, which forms the foundation of our proof.

\begin{defn}[Sublinear expander]\label{def-sublinear-expander}
	For each $\ep_{1}>0$ and $k>0$, a graph $G$ is an \emph{$(\ep_{1},k)$-expander} if
	$$\abs{N(X)}\ge \ep(\abs{X},\ep_1,k)\cdot \abs{X}$$
	for all $X\subseteq V(G)$ of size $k/2\le \abs{X}\le |V(G)|/2$, where
	$$\ep(x,\ep_{1},k):=
	\begin{cases}
		0 & \text{if }x<k/5,\\
		\ep_{1}/\log^2(15x/k) & \text{if }x\ge k/5.
	\end{cases}$$
\end{defn}
Whenever the choices of $\ep_{1},k$ are clear, we omit them and write $\ep(x)$ for $\ep(x,\ep_{1},k)$. Note that $\ep(x,\ep_{1},k)$ decreases as $x$ increases when $x\ge k/2$. Koml\'os and Szemer\'edi \cite{K-Sz-2} showed that every graph $G$ contains a sublinear expander as dense as $G$.

\begin{theorem}[\cite{K-Sz-2}]\label{thm: expander-lemma}
	There exists $\ep_{1}>0$ such that the following holds for every $k>0$. Every graph $G$ has an $(\ep_1 ,k)$-expander $H$ with $d(H)\ge d(G)/2$ and $\delta(H)\ge d(H)/2$.
\end{theorem}

Note that, in Theorem \ref{thm: expander-lemma}, the sublinear expander $H$ can be much smaller than the original graph $G$. Indeed, $G$ could be the disjoint union of many copies of such a graph $H$.

The following lemma is the key property of sublinear expanders that we will use.  It roughly says that in a sublinear expander, we can connect two sets of vertices using a short path while avoiding a smaller vertex set.

\begin{lemma}[\cite{K-Sz-2}]\label{lem: small-diameter-lemma}
    Let $\ep_{1},k>0$. If $G$ is an $n$-vertex $(\ep_{1},k)$-expander, then any two vertex sets, each of size at least $x\ge k$, are of distance at most $\frac{2}{\ep_{1}}\log^3(\frac{15n}{k})$ apart. This remains true even after deleting $x\cdot \ep(x)/4$ arbitrary vertices from $G$.
\end{lemma}

It is convenient to work on a bipartite graph, so we use the following well known result.

\begin{prop}\label{prop: bipartite}
	Every graph $G$ contains a bipartite subgraph $H$ with $d(H)\ge d(G)/2$.
\end{prop}

Combining this proposition with Theorem \ref{thm: expander-lemma}, we get the following corollary immediately.

\begin{cor}\label{cor: bipartite expander}
	There exists $\ep_{1}>0$ such that the following holds for every $k>0$ and $d\in \mathbb{N}$. Every graph $G$ with $d(G)\ge 8d$ has a bipartite $(\ep_{1},k)$-expander $H$ with $\delta(H)\ge d$.
\end{cor}

The following proposition shows that every $K_{s,t}$-free $(\ep_1,\ep_2d^{s/(s-1)})$-expander is an $(\ep_1,\ep_2d)$-expander.

\begin{prop}[Proposition 5.2 in \cite{LM1}]\label{prop: d2-d expander}
	Let $0<\ep_1<1$, $0<\ep_2<1/10^5t$, and let $t\ge s\ge 2$ be integers. If $G$ is a $K_{s,t}$-free,  $(\ep_1,\ep_2d^{s/(s-1)})$-expander with $\delta(G)\ge d/16$, then $G$ is also an $(\ep_1,\ep_2d)$-expander.
\end{prop}

\subsection{Main results}

The following is a rough proof strategy for Theorem \ref{thm: balanced subdivision} and Theorem \ref{thm: c4-free balanced subdivision}. By Corollary \ref{cor: bipartite expander}, $G$ contains a bipartite subgraph which has some expansion properties. Then depending on whether the subgraph is dense or not, we divide the proof into two cases. The dense case is handled in Lemma \ref{thm: dense}, and the sparse case is covered in Lemma \ref{lem: sparse} \cite{WY}.

\begin{lemma}\label{thm: dense}
	There exists $\ep_1>0$ such that, for every $0<\ep_2<1/5$ and $s\ge 240$, there exist $d_0$ and some constant $c>0$ such that the following holds for each $n\geq d\geq d_0$ and $d\ge \log^sn$. \\
    {\rm(\romannumeral1)} If $G$ is a bipartite n-vertex $(\ep_{1},\ep_{2}d)$-expander with $\delta(G)\ge d$, then $G$ contains a $\mathsf{T}K_{c\sqrt{d}}^{(z)}$ for some $z\in \mathbb{N}$; \\
    {\rm(\romannumeral2)} If $G$ is a $C_4$-free bipartite n-vertex $(\ep_{1},\ep_{2}d^2)$-expander with $\delta(G)\ge d$, then $G$ contains a $\mathsf{T}K_{cd}^{(z)}$ for some $z\in \mathbb{N}$.
\end{lemma}

\begin{lemma}[Lemma 1.3 in \cite{WY}]\label{lem: sparse}
	There exists $\ep_1>0$ such that for any $0<\ep_2<1/5$ and $s\ge 20$, there exist $d_0$ and some constant $c>0$ such that the following holds for each $n\ge d\ge d_0$ and $d< \log^sn$. Suppose that $G$ is a $\mathsf{T}K_{d/2}^{(2)}$-free bipartite n-vertex $(\ep_{1},\ep_{2}d)$-expander with $\delta(G)\ge d$. Then $G$ contains a $\mathsf{T}K_{cd}^{(z)}$ for some $z\in \mathbb{N}$.
\end{lemma}

We will show that Theorem \ref{thm: balanced subdivision} and Theorem \ref{thm: c4-free balanced subdivision} follow from Lemma \ref{thm: dense} and Lemma \ref{lem: sparse}.

\begin{proof}[Proof of Theorem~\ref{thm: balanced subdivision}.]
	Given $\ep_2=1/10$ and $s=240$, we have constants $\ep_1$ and $d_0$ such that the properties in Corollary \ref{cor: bipartite expander}, Lemma \ref{thm: dense} and Lemma \ref{lem: sparse} hold. Let $G$ be a graph with average degree $d(G)=d$ for some $d\ge d_0$. Write $d_1=d/8$. By Corollary \ref{cor: bipartite expander} with $d=d_1$, $G$ has a bipartite $(\ep_{1},\ep_{2}d_1)$-expander $H$ with $\delta(H)\ge d_1$.
	Let $n=\abs{H}$. If $d_1\ge \log^sn$, then by Lemma \ref{thm: dense}, $G$ contains a $\mathsf{T}K_{c_1\sqrt{d_1}}^{(z)}$ for some constant $c_1>0$ and some $z\in \mathbb{N}$. Otherwise, by Lemma \ref{lem: sparse}, $G$ contains either a $\mathsf{T}K_{d_1/2}^{(2)}$ or a $\mathsf{T}K_{c_2d_1}^{(z)}$ for some constant $c_2>0$ and some $z\in \mathbb{N}$. This finishes the proof by taking $c=\min\{1/16,\sqrt{2}c_1/4, c_2/8\}$.
\end{proof}

\begin{proof}[Proof of Theorem~\ref{thm: c4-free balanced subdivision}.]
	Fix $\ep_2 =1/10^6$ and $s=240$, there exist constants $\ep_1$ and $d_0$ such that the conclusions of Corollary \ref{cor: bipartite expander}, Lemma \ref{thm: dense} and Lemma \ref{lem: sparse} hold. Let $G$ be a $C_4$-free graph with average degree at least $d$ for some $d\geq d_0$. Write $d_1=d/8$. By Corollary \ref{cor: bipartite expander} with $d=d_1$, $G$ has a bipartite $(\ep_{1},\ep_{2}d_1^2)$-expander $H$ with $\delta(H)\ge d_1$. Let $n=\abs{H}$. If $d_1\ge \log^sn$, then by Lemma \ref{thm: dense}, $G$ contains a $\mathsf{T}K_{c_1d_1}^{(z)}$ for some constant $c_1>0$ and some $z\in \mathbb{N}$. Otherwise, by Proposition \ref{prop: d2-d expander} with $s=t=2$, $H$ is also an $(\ep_{1},\ep_{2}d_1)$-expander. Then by Lemma \ref{lem: sparse}, $G$ contains either a $\mathsf{T}K_{d_1/2}^{(2)}$ or a $\mathsf{T}K_{c_2d_1}^{(z)}$ for some constant $c_2>0$ and some $z\in \mathbb{N}$. This finishes the proof by taking $c=\min\{1/16,c_1/8, c_2/8\}$.
\end{proof}

\section{Proof of Lemma \ref{thm: dense}}\label{sec:proof1}

In this section we prove Lemma \ref{thm: dense}. To achieve this, we first introduce some structures from Liu and Montgomery \cite{LM1,LM2}.

\subsection{Gadgets}

\begin{defn}[Hub \cite{LM1}]\label{hub}
	Given integers $h_1,h_2>0$, an \emph{$(h_1,h_2)$-hub} is a graph consisting of a center vertex $u$, a set $S_{1}(u)\subseteq N(u)$ of size $h_1$, and pairwise disjoint sets $S_1(z)\subseteq N(z)\backslash \{u\}$ of size $h_2$ for each $z\in S_{1}(u)$. Denote by $H(u)$ a hub with center vertex $u$ and write $B_1(u)=\{u\}\cup S_{1}(u)$ and $S_2(u)=\bigcup_{z\in S_{1}(u)}S_{1}(z)$. For any $z\in S_{1}(u)$, write $B_1(z)=\{z\} \cup S_{1}(z)$.
\end{defn}

\begin{defn}[Unit \cite{LM1}]\label{unit}
	Given integers $h_0,h_1,h_2,h_3>0$, an \emph{$(h_0,h_1,h_2,h_3)$-unit} $F$ is a graph consisting of a core vertex $v$, $h_0$ vertex-disjoint $(h_1,h_2)$-hubs $H(u_1),\cdots,H(u_{h_0})$ and pairwise disjoint $v,u_j$-paths of length at most $h_3$. By the \emph{exterior} of the unit, denoted by Ext($F$), we mean $\bigcup_{j=1}^{h_0}S_{2}(u_j)$. Denote by Int($F$):=$V(F)\backslash $Ext($F$) the \emph{interior} of the unit.
\end{defn}

\begin{figure}[htbp]\label{L1}
	\centering
	\includegraphics[scale=0.4]{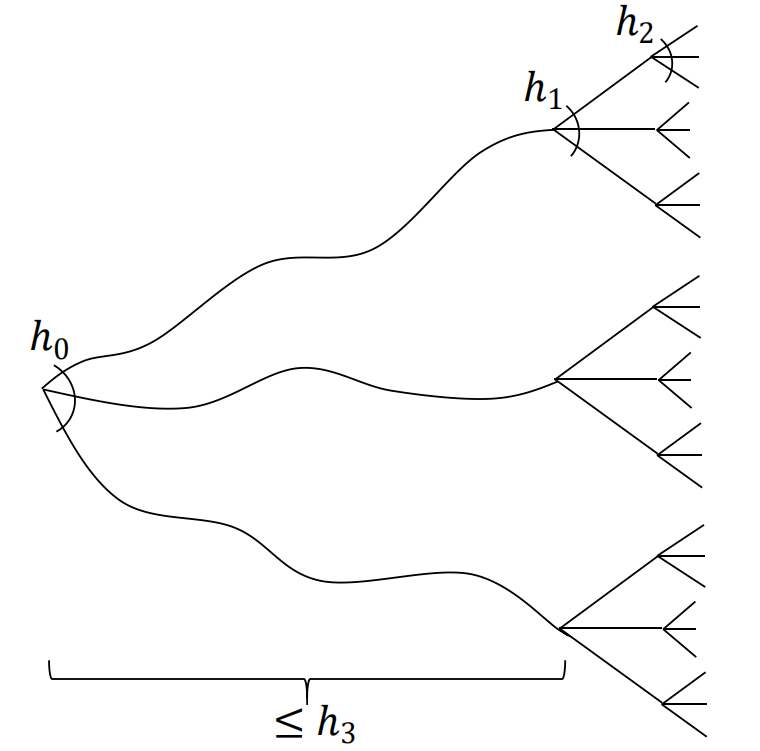}
	\caption{$(h_0,h_1,h_2,h_3)$-unit}
\end{figure}

\begin{defn}[Expansion \cite{LM2}]\label{expansion}
	Given a vertex $v$ in a graph $F$, $F$ is a \emph{$(D,m)$-expansion} of $v$ if $\abs{F}=D$ and $v$ is at distance at most $m$ in $F$ from any other vertex of $F$.
\end{defn}

By the definition of expansion, we have the following property.

\begin{prop}[\cite{LM2}]\label{prop: expansion}
	Let $D,m\in \mathbb{N}$ and $1\le D'\le D$. Then, any graph $F$ which is a $(D,m)$-expansion of $v$ contains a subgraph which is a $(D',m)$-expansion of $v$.
\end{prop}

Liu and Montgomery \cite{LM2} introduced a structure called \emph{adjuster} which contains a collection of paths whose lengths form a long arithmetic progression of difference 2. We shall use this structure to adjust a path to a desired length.

\begin{defn}[Adjuster \cite{LM2}]\label{adjuster}
	A \emph{$(D,m,k)$-adjuster} $\mathcal{A}=(v_1,F_1,v_2,F_2,A)$ in a graph $G$ consists of \emph{core vertices} $v_1,v_2\in V(G)$, graphs $F_1,F_2\subseteq G$ and a \emph{center vertex set} $A\subseteq V(G)$ such that the following hold for some $l\in \mathbb{N}$.
	\begin{itemize}
	\item[$\mathbf{A1}$] $A, V(F_1)$ and $V(F_2)$ are pairwise disjoint.
	\item[$\mathbf{A2}$] For each $i\in [2]$, $F_i$ is a $(D,m)$-expansion of $v_i$.
	\item[$\mathbf{A3}$] $\abs{A} \leq 10mk$.
	\item[$\mathbf{A4}$] For each $i\in \{0,1,\cdots,k\}$, there is a $v_1,v_2$-path in $G[A\cup \{v_1,v_2\}]$ of length $l+2i$.
\end{itemize}
\end{defn}

We refer to the subgraphs $F_1$ and $F_2$ as the \emph{ends} of the adjuster, and note that $V(\mathcal{A})=V(F_1)\cup V(F_2)\cup A$. We denote by $l(\mathcal{A})$ the smallest such $l$ for which $\mathbf{A4}$ hold. Then it immediately follows that $l(\mathcal{A})\le \abs{A}+1\le 10mk+1$. We call a $(D,m,1)$-adjuster a \emph{simple adjuster}.

\subsection{Proof of Lemma~\ref{thm: dense}}\label{subsec: dense}

Throughout the paper, we always choose $m$ to be the smallest even integer which is larger than $80\log^{4}\frac{n}{d}$ (or $80\log^{4}\frac{n}{d^2}$ for the $C_4$-free  case) in which $\delta(G)\ge d$. We first introduce the following two lemmas. The first one is from Fox and Sudakov \cite{F-S}, which states the existence of a balanced clique subdivision in a dense graph.

\begin{lemma}[Theorem 8.1 in \cite{F-S}]\label{lem: fox-sudakov}
    For every $\ep >0$, if $G$ is a graph with $n$ vertices and $\ep n^2$ edges, then $G$ contains a $\mathsf{T}K_{\ep \sqrt{n}}^{(2)}$.
\end{lemma}

The second lemma is from Balogh, Liu and Sharifzadeh \cite{BLS} on large balanced clique subdivisions in dense $C_4$-free graphs.

\begin{lemma}[Theorem 1.4 in \cite{BLS}]\label{lem: c4-free dense}
    For every $c>0$ there is a $c'>0$ such that the following holds. If $G$ is a $C_4$-free graph with $n$ vertices and $cn^{3/2}$ edges, then $G$ contains a $\mathsf{T}K_{c'\sqrt{n}}^{(4)}$.
\end{lemma}

Then based on the two lemmas above, it remains to consider the case when $d$ is much less than $n$. An outline of the proof of Lemma \ref{thm: dense} is as follows. We construct a collection of units whose interiors are pairwise disjoint (see Lemma \ref{lem: robust unit}), and the core vertices of those units would serve as the branch vertices of a balanced clique subdivision. In order to connect a pair of core vertices of two units, $v_1,v_2$ say, by a path of fixed length $z$, we first construct an adjuster. Adapting an approach in Liu and Montgomery \cite{LM2}, we also need to robustly build a desired adjuster by linking simple adjusters. Unlike in \cite{LM2} and \cite{WY} where the simple adjuster is relatively small, we need a large-sized one for our purpose (see Lemma \ref{lem: robust adjuster}). To be more precise, we construct an adjuster of order roughly $dm^{O(1)}$. To achieve this, we shall introduce a specific structure (see Definition \ref{octopus}).
Then we find two paths connecting the exteriors of two units to two expansions of the adjuster. By extending those two paths to $v_1$, $v_2$ within two units respectively, we can get two paths $P$ and $Q$ such that the sum of their length is close to $z$. Finally, by the property of the adjuster, we obtain an intermediate path $R$ of length $z-l(P)-l(Q)$, and thus $P\cup R\cup Q$ is a path as desired.

The following lemma enables us to build a unit whilst avoiding any medium-sized vertex set. This helps us construct many units whose interiors are pairwise disjoint.

\begin{lemma}\label{lem: robust unit}
    For any $0<\ep_1,\ep_2<1$, $s\ge 240$ and let $c=1/200$, there exists $K$ such that the following holds for sufficiently large $n$ and $d$ with $d\ge \log^{s}n$. \\
    {\rm(\romannumeral1)} Let $G$ be a $\mathsf{T}K_{\sqrt{d}}^{(2)}$-free bipartite n-vertex $(\ep_{1},\ep_{2}d)$-expander with $\delta(G)\ge d$ and $n\ge Kd$. Given any set $W\subseteq V(G)$ with size at most $2cdm^{4}$, we have that $G-W$ contains a $(c\sqrt{d},m^4,c\sqrt{d},2m)$-unit. \\
    {\rm(\romannumeral2)} Let $G$ be a $C_4$-free bipartite n-vertex $(\ep_{1},\ep_{2}d^2)$-expander with $\delta(G)\ge d$ and $n\ge Kd^2$. Given any set $W\subseteq V(G)$ with size at most $2cd^2m^{4}$, we have that $G-W$ contains a $(cd,m^4,cd,2m)$-unit.
\end{lemma}

We shall make use of the following result to build a desired adjuster robustly.

\begin{lemma}\label{lem: robust adjuster}
	There exists some $\ep_1>0$ such that, for every $0<\ep_2<1/5$ and integer $s\ge 240$, there exists $d_0$ and $K$ such that the following holds for each $n\ge d\ge d_0$ and $d\ge \log^sn$. \\
	{\rm(\romannumeral1)} Let $G$ be a $\mathsf{T}K_{\sqrt{d}}^{(2)}$-free n-vertex $(\ep_{1},\ep_{2}d)$-expander with $\delta(G)\ge d$ and $n\ge Kd$. Let $D=dm^4/10^7$. Let $W\subseteq G$ satisfy $\abs{W}\le D/\log^3\frac{n}{d}$. Then $G-W$ contains a $(D,m,r)$-adjuster for any $r\le \frac{1}{10}dm^2$. \\
	{\rm(\romannumeral2)} Let $G$ be a $C_4$-free n-vertex $(\ep_{1},\ep_{2}d^2)$-expander with $\delta(G)\ge d$ and $n\ge Kd^2$. Let $D=d^2m^4/10^7$. Let $W\subseteq G$ satisfy $\abs{W}\le D/\log^3\frac{n}{d^2}$. Then $G-W$ contains a $(D,m,r)$-adjuster for any $r\le \frac{1}{10}d^2m^2$.
\end{lemma}

The following lemma helps us find two paths respectively connecting two pairs of vertex sets whilst avoiding a smaller vertex set.

\begin{lemma}\label{cor: connect four}
	For any $0<\ep_1,\ep_2<1$, there exists $d_0$ and $K$ such that the following holds for each $n\ge d\ge d_0$ and $d\ge \log^sn$. \\
	{\rm(\romannumeral1)} Let $G$ be a $\mathsf{T}K_{\sqrt{d}}^{(2)}$-free n-vertex $(\ep_{1},\ep_{2}d)$-expander with $\delta(G)\ge d$ and $n\geq Kd$. Let $D=dm^4/10^7$, $l\le dm^3$. Let $W\subseteq V(G)$ satisfy $\abs{W}\le D/\log^3\frac{n}{d}$, or \\
	{\rm(\romannumeral2)} Let $G$ be a $C_4$-free n-vertex $(\ep_{1},\ep_{2}d^2)$-expander with $\delta(G)\ge d$ and $n\geq Kd^2$. Let $D=d^2m^4/10^7$, $l\le d^2m^3$. Let $W\subseteq V(G)$ satisfy $\abs{W}\le D/\log^3\frac{n}{d^2}$. \\
	Let $U_i\subseteq V(G)-W$ be disjoint vertex sets of size at least $D$, $i\in \{1,2\}$, and $F_j\subseteq G-W-U_1-U_2$ be vertex-disjoint $(D,m)$-expansion of $v_j$, $j\in \{3,4\}$. Then, $G-W$ contains vertex-disjoint paths $P$ and $Q$ with $l\le l(P)+l(Q)\le l+13m$ such that both $P$ and $Q$ connect $\{v_1,v_2\}$ to $\{v_3,v_4\}$ for some $v_i\in U_i$ with $i\in \{1,2\}$.
\end{lemma}

Throughout the rest of this paper, we always choose
\begin{equation}
\kappa=\sqrt{d} \ ({\rm or\ } d {\ \rm for\ } C_4\text{-}{\rm free\ case})\label{kappa}
\end{equation}
and recall that $m$ is the smallest even integer which is larger than $80\log^{4}\frac{n}{\kappa^2}$. As $n/\kappa^2\ge K$, when $K$ is sufficiently large, we obtain that $n/\kappa^2$ and also $m$ are sufficiently large, and
\begin{equation}
n/\kappa^2 \ge m^a\label{mm}
\end{equation}
for any given constant $a$. Recall that $\ep(x)$ is decreasing, and since $n/\kappa^2\ge K$ is sufficiently large,
\begin{equation}\ep(n)=\frac{\ep_1}{\log ^2(15n/\ep_2\kappa^2)}\ge \frac{10^5}{\log^3(n/\kappa^2)}\ge \frac{10^5}{m}.\label{unit5}
\end{equation}

Now we are ready to prove Lemma \ref{thm: dense}.

\begin{proof}[Proof of Lemma~\ref{thm: dense}.]
    The proof ideas of (\romannumeral1) and (\romannumeral2) are similar, and we choose $\kappa$ as in \eqref{kappa}. Note that we shall assume that $G$ is $\mathsf{T}K_{\kappa}^{(2)}$-free for (\romannumeral1), otherwise the proof is done. Let $c'=1/200$ and $z=m^3$. We choose $K$ and $d_0$ to be sufficiently large.
	
	For (\romannumeral1), if $n<K\kappa^2$, then by Lemma \ref{lem: fox-sudakov} with $\ep=1/K$, $G$ contains a balanced subdivision of a clique of size $\ep \sqrt{n}$. For (\romannumeral2), if $n<K\kappa^2$, then by Lemma \ref{lem: c4-free dense} with $c_1=1/K$, there exists $c_1'>0$ such that $G$ contains a balanced subdivision of a clique of size $c_1' \sqrt{n}$. Now it remains to consider the case $n/\kappa^2\ge K$. Since $n/\kappa^2\ge K$, we have $m\ge 80\log^4K$. We claim that we can greedily pick in $G$ a collection $\{M_1,\cdots,M_{c'\kappa}\}$ of $(c'\kappa,m^4,c'\kappa,2m)$-units with disjoint interiors. Indeed, this is possible by applying Lemma \ref{lem: robust unit} on $G$ with $W$ being the interiors of $(c'\kappa,m^4,c'\kappa,2m)$-units we have constructed in each stage of the process, and by the fact that the size of interiors of $c'\kappa$ such units is at most $c'\kappa(2m+1+m^4)\cdot c'\kappa\le 2c'^2\kappa^2m^4$.
	By the pigeonhole principle, we can find $c'\kappa/2$ such units among them such that their core vertices are in the same part of the bipartition for $G$. Without loss of generality, these units are $M_1,\cdots,M_{c'\kappa/2}$ with core vertices $w_1,\cdots,w_{c'\kappa/2}$ and denote by $u_{i,j}$ the center of the $j$-th hub in $M_i$, where $1\le i\le c'\kappa/2$ and $1\le j\le c'\kappa$. Let $W$ be the union of the vertices in the $w_i,u_{i,j}$-paths in all the units, including their endvertices. Then $\abs{W}\le c'\kappa\cdot (2m+1)\cdot c'\kappa/2\le 2c'^2\kappa^2 m$.
	We will construct a $\mathsf{T}K_{c'\kappa/4}^{(z)}$ as follows.
	Let $\mathcal{P}$ be a maximum collection of paths under the following rules.

    \begin{itemize}
	\item[$\mathbf{B1}$] Each path connects one pair of center vertices of hubs from different units, such that it can be extended to a path of length $z$ connecting the core vertices of those two units.
	\item[$\mathbf{B2}$] All paths in $\mathcal{P}$ are pairwise disjoint, and the internal vertices of those paths are disjoint from $W$.
	\item[$\mathbf{B3}$] For each pair of units, there is at most one path in $\mathcal{P}$ between their respective hubs.
    \end{itemize}

    Let $W_1$ be the set of vertices of $\mathcal{P}$. Then by $\mathbf{B1}$ and $\mathbf{B3}$, $\abs{W_1}\le (m^3+1)\tbinom{c'\kappa/2}{2}\le c'^2\kappa^2m^3$. Let us call a unit \emph{bad} if more than $\kappa m^3$ vertices in the interior of this unit have been used in $\mathcal{P}$, and \emph{good} otherwise. Thus, as the interior of the units are disjoint, there are at most $\frac{c'^2\kappa^2m^3}{\kappa m^3}\le c'\kappa/4$ bad units.

    \begin{claim}\label{claim: connect units}
    For every pair of good units, there is a path in $\mathcal{P}$ between two of their respective hubs.
    \end{claim}
    \begin{proof}
    Without loss of generality, we may assume for contradiction that $M_1$ and $M_2$ are a pair of good units for which there is no desired path in $\mathcal{P}$. Let $B$ be the union of the interiors of $M_1$ and $M_{2}$. Then we have $\abs{B}\le 2\cdot 2c'\kappa m^4$. Set $W'=W\cup W_1\cup B$. Note that $\abs{W'}\le 2c'^2\kappa^2m+c'^2\kappa^2m^3+4c'\kappa m^4\le 3c'^2\kappa^2m^3$.
    By $\mathbf{B3}$, $M_1$ (or $M_2$) has at least $c'\kappa/2$ hubs whose center vertices are not used in $\mathcal{P}$. Denote these hubs by $H(u_{1,k})$, $k\in I$ for some index set $I$ with $\abs{I}\ge c'\kappa/2$. Let $A_1$ be the set of vertices in $\cup_{k\in I}S_1(u_{1,k})$ not used in $\mathcal{P}$. Similarly, we define $A_2$ for $M_2$. As $M_1$ is good, we have

    $$\abs{A_1}\ge \cup_{k\in I}S_1(u_{1,k})-\kappa m^3\ge \frac{c'\kappa}{2} m^4-\kappa m^3\ge \frac{c'\kappa m^4}{4}.$$
    The last inequality holds as $m\ge 4/c'$ for sufficiently large $K$. As the hubs in $M_1$ are disjoint, we have

    $$\abs{N_{M_1}(A_1)\backslash W'}\ge \frac{c'\kappa m^4}{4} c'\kappa-3c'^2\kappa^2m^3\ge \frac{c'^2\kappa^2m^4}{8}.$$
    Similarly, $\abs{N_{M_{2}}(A_{2})\backslash W'}\ge c'^2\kappa^2m^4/8$.

    Set $D=\kappa^2m^4/10^7$. Recall that $n\ge K\kappa^2$ and $m$ is the smallest even integer which is larger than $80\log^{4}\frac{n}{\kappa^2}$. As $d\ge\log^sn\ge m^{60}$, for sufficiently large $K$, we have $\abs{W'}\le 3c'^2\kappa^2m^3 \le D/2\log^3\frac{n}{\kappa^2}$. By Lemma \ref{lem: robust adjuster} with $W=W'$, there is a $(D,m,21m)$-adjuster in $G-W'$, which is denoted by $\mathcal{A}=(v_1,F_1,v_2,F_2,A)$. By the definitions of expansion and adjuster, we have $\abs{A}\le 210m^2$, $l(\mathcal{A})\le \abs{A}+1\le 220m^2$ and $\abs{V(F_1)}=\abs{V(F_2)}=D$. Let $l'=z-21m-l(\mathcal{A})$. We have $0\le l'\le \kappa^2m^3$. Let $U_1'=N_{M_1}(A_1)\backslash (W'\cup V(F_1)\cup V(F_2))$ and $U_2'=N_{M_{2}}(A_{2})\backslash (W'\cup V(F_1)\cup V(F_2))$.
    As the size of $U_1'$ (or $U_2'$) is at least $c'^2\kappa^2m^4/8-2D\ge 2D$, there are disjoint vertex sets $U_1\subseteq U_1'$ and $U_2\subseteq U_2'$ such that $\abs{U_i}\ge D$, $i\in \{1,2\}$. As $d\ge\log^sn\ge m^{60}$, $\abs{A\cup W'}\le 210m^2+D/2\log^3\frac{n}{\kappa^2}\le D/\log^3\frac{n}{\kappa^2}$. By Lemma \ref{cor: connect four}, there are vertex-disjoint paths $P_1$ and $Q_1$ in $G-A-W'$ connecting $\{u_1,u_2\}$ to $\{v_1,v_2\}$ for some  $u_i\in U_i$, $i\in \{1,2\}$, and $l'\le l(P_1)+l(Q_1)\le l'+13m$. Without loss of generality, we can assume that $P_1$ is a $u_1,v_1$-path and $Q_1$ is a $u_2,v_2$-path.

    As $U_1$ is a subset of the exterior of $M_1$, there exists a path $P_2$ from $w_1$ to $u_1$ of length at most $2m+2\le 4m$ in $M_1$. Similarly, there exists a path $Q_2$ from $w_{2}$ to $u_2$ of length at most $2m+2\le 4m$ in $M_{2}$. Let $P=P_1\cup P_2$ and $Q=Q_1\cup Q_2$. Thus, $P$ is a $w_1,v_1$-path and $Q$ is a $w_{2},v_2$-path, with $l'\le l(P)+l(Q)\le l'+21m$.

    Now, $l(\mathcal{A})\le z-l(P)-l(Q)\le l(\mathcal{A})+21m$. As $G$ is a bipartite graph and $w_1$, $w_{2}$ are in the same part, $l(\mathcal{A})$ and $z-l(P)-l(Q)$ have the same parity. Thus there is a $v_1,v_2$-path in $G[A\cup \{v_1,v_2\}]$ of length $z-l(P)-l(Q)$, $R$ say, such that $P\cup R\cup Q$ is a $w_1,w_{2}$-path of length $z$ which satisfies $\mathbf{B1}$-$\mathbf{B3}$, contradicting the maximality of $\mathcal{P}$.
    \end{proof}
    Finally, for every pair of good units $M_i,M_j$, there exists a $w_i,w_j$-path of length $z$, and by $\mathbf{B2}$ these paths are disjoint outside of their endvertices. As there are at least $c'\kappa/4$ units are good, we have a $\mathsf{T}K_{c'\kappa/4}^{(z)}$. This finishes the proof by taking $c=\min\{c'/4,\ep,c_1'\}$.
\end{proof}

\section{Proof of main tools}\label{sec:dense proof}

\subsection{Constructing units}\label{subsec: robust unit}

In this section, we find a collection of units whose interiors are pairwise disjoint, and we prove this by iteratively constructing a unit whilst avoiding the interiors of previous units. To achieve this, we first prove that every $\mathsf{T}K_{\sqrt{d}}^{(2)}$-free (or $C_4$-free) graph maintains its average degree while deleting a vertex set in Lemma \ref{lem: robust degree} (or in Corollary \ref{cor: c4-free robust degree}). Then by applying Lemma \ref{lem: robust hub} (or Lemma \ref{lem: c4-free robust hub}), we can greedily pick a collection of vertex-disjoint hubs of certain types, and connect them with internally vertex-disjoint short paths such that one of the hubs would be linked to many others (see Claim \ref{claim: find unit}), forming the desired unit.

We first show that a $\mathsf{T}K_{\sqrt{d}}^{(2)}$-free graph keeps its average degree while deleting a medium-sized vertex set.

\begin{lemma}\label{lem: robust degree}
    Let $s\ge 8x>0$. There exists $K=K(s,x)$ such that the following holds for each $n$ and $d$ satisfying $n\ge Kd$ and $d\ge \log^{s}n$. If $G$ is a $\mathsf{T}K_{\sqrt{d}}^{(2)}$-free $n$-vertex graph with $\delta(G)\ge d$, then for any vertex set $W\subseteq V(G)$ of size at most $dm^x$, we have $d(G-W)\ge \frac{d}{2}$.
\end{lemma}

\begin{proof}
	It is easy to see this lemma holds when $\abs{W}\le \frac{d}{2}$. Suppose to the contrary that there exists some $W$ such that $d(G-W)< \frac{d}{2}$ when $\abs{W}> \frac{d}{2}$. By \eqref{mm}, we have $\abs{W}\le dm^x \le n/2$, then
	$$e(V(G-W),W)= \sum\limits_{v\in V(G-W)}d_{G}(v)-2e(G-W)\ge \frac{d}{2}\abs{G-W}.$$
	By Lemma \ref{depend} with $(V_1, V_2,a,c,r)=(W,G-W,\sqrt{d},d,2)$ and
	$$\alpha=\frac{e(V(G-W),W)}{\abs{G-W}\abs{W}}\ge \frac{\frac{d}{2}\abs{G-W}}{\abs{G-W}\cdot dm^x}=\frac{1}{2m^x},$$
	if there exists some $t$ such that
	$$\left(\frac{1}{2m^x}\right)^t\cdot\frac{d}{2}-\binom{dm^x}{2}\left(\frac{d}{n/2}\right)^t\ge \sqrt{d},$$
	then $G$ contains a $\mathsf{T}K_{\sqrt{d}}^{(2)}$, which is a contradiction. We take $t=\log_{8m^x}\sqrt{d}$. Note that $t\ge 1$ as $s\ge 8x$. Then we have
	\begin{equation}
	\left(\frac{1}{2m^x}\right)^t\cdot\frac{d}{2}\ge \frac{4}{(8m^x)^t}\cdot\frac{d}{2}=2\sqrt{d},\label{degree1}
	\end{equation}
	and
	\begin{equation}
	\binom{dm^x}{2}\left(\frac{d}{n/2}\right)^t\le \frac{d^2m^{2x}}{(\frac{n}{2d})^t}\le \frac{d^2m^{2x}}{(8m^x)^{8t}}\le \frac{\log^{8x}n}{\log^{2s}n}\le 1.\label{degree2}
	\end{equation}
	The second inequality in \eqref{degree2} holds by \eqref{mm}.
	By \eqref{degree1} and \eqref{degree2}, we have
	$$\left(\frac{1}{2m^x}\right)^t\cdot\frac{d}{2}-\binom{dm^x}{2}\left(\frac{d}{n/2}\right)^t> \sqrt{d}.$$
	The proof is complete.
\end{proof}

Then we show that a $C_4$-free graph maintains its average degree while deleting a vertex set of moderate size through the following result of K\H{o}v{\'a}ri, S{\'o}s and Tur{\'a}n \cite{KST}.

\begin{lemma}[\cite{KST}]\label{lem:kst}
    Let $G=(A,B)$ be a bipartite graph that does not contain a copy of $K_{s,t}$ with $t$ vertices in $A$ and $s$ vertices in $B$. Then
    \begin{equation*}
        \left| A\right| \binom{\overline{d}(A)}{s}\leq t\binom{\left| B\right| }{s},
    \end{equation*}
    where $\overline{d}(A)=\sum_{v\in A}\frac{d(v)}{\left| A\right| }$ is the average degree in $G$ of the vertices in $A$.
\end{lemma}

Thus we can get the following corollary.

\begin{cor}\label{cor: c4-free robust degree}
    Let $s>0$ and $x>0$. There exists $K$ such that the following holds for each $n$ and $d$ satisfying $n\ge Kd^2$ and $d\ge \log^{s}n$. If $G$ is a $C_4$-free $n$-vertex graph with $\delta(G)\ge d$, then for any vertex set $W\subseteq V(G)$ of size at most $d^2m^x$, we have $d(G-W)\ge \frac{d}{2}$.
\end{cor}

\begin{proof}
	Let $H=(V(G-W),W,E)$ be a bipartite subgraph of $G$, and $E$ be the set of all edges between $V(G-W)$ and $W$ in $G$. By Lemma \ref{lem:kst} with $(G,A,B,s,t)=(H,V(G-W),W,2,2)$, we have
	\begin{equation*}
		\left| G-W\right| \binom{\overline{d}(V(G)-W)}{2}\leq 2\binom{\left| W\right| }{2}.
	\end{equation*}
	Therefore
	\begin{equation*}
		\overline{d}(V(G)-W)\leq \frac{\sqrt{2}\left| W\right| }{\sqrt{\left| G-W\right| }}+1\leq \frac{2d^2m^x}{\sqrt{n}}\leq \frac{d}{2},
	\end{equation*}
	where the third inequality follows from \eqref{mm} and the choice of $K$. Thus we have $d(G-W)\ge \frac{d}{2}$.
\end{proof}

In order to construct units, we will find many vertex-disjoint hubs, and we use the following lemma to show the existence of these hubs in a $\mathsf{T}K_{\sqrt{d}}^{(2)}$-free graph.

\begin{lemma}\label{lem: robust hub}
    Let $s\ge 8x>0$. There exists $K$ such that the following holds for each $n\ge Kd$, $d\ge \log^{s}n$ and any $h_1,h_2\le \sqrt{d}/100$. If $G$ is a $\mathsf{T}K_{\sqrt{d}}^{(2)}$-free n-vertex graph with $\delta(G)\ge d$, then for any set $W\subseteq V(G)$ with size at most $dm^{x}$, we have that $G-W$ contains an $(h_1,h_2)$-hub.
\end{lemma}

\begin{proof}
    Let $K$ be sufficiently large. By Lemma \ref{lem: robust degree}, we have $d(G-W)\ge d/2$, so there exists a subgraph $H\subseteq G-W$ with $\delta (H)\ge d/4$. We choose an arbitrary vertex $v$ in $H$. As the size of the vertex set of $(h_1,h_2)$-hub is $1+h_1+h_1\cdot h_2\le 2h_1\cdot h_2\le d/4$, we can greedily find a hub with center vertex $v$ as desired.
\end{proof}

We can also robustly find hubs in a $C_4$-free graph.

\begin{lemma}\label{lem: c4-free robust hub}
	Let $s>0$ and $x>0$. There exists $K$ such that the following holds for each $n\ge Kd^2$, $d\ge \log^{s}n$ and any $h_1,h_2\le d/100$. If $G$ is a $C_4$-free n-vertex graph with $\delta(G)\ge d$, then for any set $W\subseteq V(G)$ with size at most $d^2m^{x}$, we have that $G-W$ contains an $(h_1,h_2)$-hub.
\end{lemma}

\begin{proof}
	Let $K$ be sufficiently large. By Corollary \ref{cor: c4-free robust degree}, we have $d(G-W)\ge d/2$, so there exists a subgraph $H\subseteq G-W$ with $\delta (H)\ge d/4$. We arbitrarily pick a vertex $v$ in $H$. As $H$ is $C_4$-free, for any vertices $x,y\in N_H(v)$, we have $N_H(x)\cap N_H(y)=\{ v \} $. Therefore we can in a greedy way find a hub with center vertex $v$ as desired.
\end{proof}

We now expand hubs to get a unit.

\begin{proof}[Proof of Lemma~\ref{lem: robust unit}.]
	The proof ideas of (\romannumeral1) and (\romannumeral2) are similar, and we choose $\kappa$ as in \eqref{kappa}.
    We choose $K$ to be sufficiently large. Recall that $s\ge 240$, $c=1/200$ and $m$ is the smallest even integer which is larger than $80\log^{4}\frac{n}{\kappa^2}$, so $\kappa^2\ge \log^{48}n\ge m^{12}$. Since $n/\kappa^2\ge K$, we have $m\ge 80\log^4K$. For sufficiently large $K$, we have
    	\begin{equation}c\kappa\ge cm^6\ge 8m^5.\label{unit1}\end{equation}

    We claim that we can find in $G-W$ vertex-disjoint hubs $H(w_1),\cdots,H(w_{m^{6}})$ and $H(u_1),\cdots,\\ H(u_{\kappa m^{6}})$ such that each $H(w_i), 1\le i\le m^6$, is a $(2c\kappa,2c\kappa)$-hub and each $H(u_j), 1\le j\le \kappa m^6$, is a $(2m^4,2c\kappa)$-hub. Note that the total number of vertices in $W$ and all these hubs we desired is at most

    $$2c\kappa^2m^4+2(2c\kappa)^{2} m^6+2(2c\kappa)(2m^4)\kappa m^6\le 10c\kappa^2m^{10}.$$
    Note that we can find a copy of $(2m^4,2c\kappa)$-hub in a $(2c\kappa,2c\kappa)$-hub. It suffices to show that we can find a $(2c\kappa,2c\kappa)$-hub in $G$ avoiding any vertex set of size at most $10c\kappa^2m^{10}$. Indeed, this is possible by applying Lemma \ref{lem: robust hub} for (\romannumeral1) or Lemma \ref{lem: c4-free robust hub} for (\romannumeral2) with $x=10$.

    Recall that for a hub with a center vertex $v$, $S_1(v)$ is the vertex set of the neighbours of $v$ in the hub, and $B_1(v)=\{v\}\cup S_1(v)$. We will construct a unit using some vertex $w_i$ as the core vertex. Let $\mathcal{P}$ be a maximum collection of paths connecting different pairs of center vertices $\{w_i,u_j\}$ under the following rules.

    \begin{itemize}
	\item[$\mathbf{C1}$] Each path has length at most $2m$, and all the paths are pairwise internally vertex-disjoint. Moreover all the vertices of those paths are disjoint from $W$.
	\item[$\mathbf{C2}$] Each path avoids using any vertices in set $B_1(w_{i'})$ or $B_1(u_{j'})$ for $1\le i'\le m^6$ and $1\le j'\le \kappa m^6$, except for at most two vertices each in $B_1(w_i)$ and $B_1(u_j)$ when $\{w_i,u_j\}$ is the pair of vertices being connected.
    \end{itemize}
    \begin{claim}\label{claim: find unit}
    There exists a vertex $w_i$ connected to at least $c\kappa$ vertices $u_j$ via the paths in $\mathcal{P}$.
    \end{claim}
    \begin{proof}
    Suppose to the contrary that each vertex $w_i$ is connected to fewer than $c\kappa$ vertices $u_j$ by the paths in $\mathcal{P}$. Let $U$ be the set of interior vertices in all the paths in $\mathcal{P}$. Then by $\mathbf{C1}$,
    \begin{equation}\abs{U}\le 2m\cdot m^6\cdot c\kappa=2c\kappa m^7.\label{unit2}\end{equation}
    Let $W_1$ be the vertex set containing the vertices in $U$ and all the vertices in each set $B_1(u_{j})$ if $u_{j}$ has been connected to at least one of the vertices $w_i$. As there are at most $m^6\cdot c\kappa$ such vertices $u_j$, using \eqref{unit1} we have
    \begin{equation}\abs{W_1}\le \abs{U}+m^6c\kappa\cdot(2m^4+1)\overset{\eqref{unit2}}{\le} 2c\kappa m^7+4c\kappa m^{10} \overset{\eqref{unit1}}{\le} c^2\kappa^2m^5.\label{unit3}\end{equation}

    For each $1\le i\le m^6$, let $T_i=B_1(w_i)\backslash W_1$. Then by $\mathbf{C2}$ and the assumption, we have $\abs{T_i}\ge c\kappa$. As the graphs $H(w_i)$ are vertex-disjoint $(2c\kappa,2c\kappa)$-hubs, we have $\abs{\cup_{i=1}^{m^6} N_{H(w_i)}(T_i)}\ge 2c\kappa\cdot c\kappa\cdot m^6$, and hence we have
    $$\abs{\cup_{i=1}^{m^6} N_{H(w_i)}(T_i)\backslash W_1}\ge 2c\kappa\cdot c\kappa\cdot m^6-\abs{W_1}\overset{\eqref{unit3}}{\ge} c^2\kappa^2m^6.$$

    At least $\kappa m^6-m^6\cdot c\kappa\ge \kappa m^6/2$ vertices $u_j$ have not been connected by a path in $\mathcal{P}$. Without loss of generality, we write these vertices $u_1,\cdots,u_p$, where $p\ge \kappa m^6/2$. By $\mathbf{C2}$, the set $W_1$ is disjoint from $\cup_{j=1}^{p}B_1(u_j)$, we have
    $$\abs{\cup_{j=1}^{p}H(u_j)-W_1}\ge 2m^4\cdot 2c\kappa\cdot \frac{\kappa m^6}{2}-c^2\kappa^2m^5\ge c^2\kappa^2m^6.$$

    We will apply Lemma \ref{lem: small-diameter-lemma} to connect $\cup_{i=1}^{m^6} N_{H(w_i)}(T_i)\backslash W_1$ and $\cup_{j=1}^pV(H(u_j))\backslash W_1$, while avoiding the vertices in $\cup_{i=1}^{m^6} T_i$, $W$ and $W_1$. Since $d\ge\log^sn\ge m^{60}$ and $n/\kappa^2\ge K$ is sufficiently large, we have
    \begin{equation}\abs{\cup_{i=1}^{m^6} T_i}+\abs{W}+\abs{W_1}\le (2c\kappa+1)m^6+2c\kappa^2m^4+c^2\kappa^2m^5 \le 2c^2\kappa^2m^5.\label{unit4}\end{equation}
    Hence, setting $y:=c^2\kappa^2m^6$, we have
    $$\frac{1}{4}\cdot \ep(y)\cdot y\ge \frac{1}{4}\cdot \ep(n)\cdot y \overset{\eqref{unit5}}{\ge} \frac{1}{4}\cdot \frac{10^5}{m}\cdot y\ge 2c^2\kappa^2m^5.$$
    Thus, by \eqref{unit4} and Lemma \ref{lem: small-diameter-lemma}, there is a shortest path of length at most
    $$\frac{2}{\ep_1}\log^3\frac{15n}{\ep_2\kappa^2}+1\le \log^4\frac{n}{\kappa^2}\le m$$
    from some $T_i$ to some $V(H(u_j))$ avoiding $W$ and $W_1$. Taking such a path, we extend it to a $w_i,u_j$-path of length at most $1+m+2\le 2m$ in $G-(W\cup W_1)$, which together with all the paths in $\mathcal{P}$ satisfies $\mathbf{C1}$ and $\mathbf{C2}$, a contradiction.
    \end{proof}

    By Claim \ref{claim: find unit}, we have a vertex $w_i$,  which is connected to $c\kappa$ vertices $u_j$. Without loss of generality, we may take $u_1,u_2,\cdots,u_{c\kappa}$ for instance so as to ease the notation. If we can find in every $(2m^4,2c\kappa)$-hub $H(u_j)$ an $(m^4,c\kappa)$-hub which is disjoint from all the internal vertices of the $w_i,u_j$-paths in $\mathcal{P}$, $1\le j\le c\kappa$, then all such $(m^4,c\kappa)$-hubs together with all the $w_i,u_j$-paths form a $(c\kappa,m^4,c\kappa,2m)$-unit. It suffices to find for every $j\in [c\kappa]$ a $(m^4,c\kappa)$-hub centered at $u_j$ as required above.

    By $\mathbf{C2}$, for every $j\in [c\kappa]$, at most one vertex in $S_1(u_j)$ is used in the $w_i,u_j$-paths. Denote by $W_2$ the vertices in all the $w_i,u_j$-paths, $1\le j\le c\kappa$. By $\mathbf{C1}$, we have that $\abs{W_2}\le 2c\kappa m$. Therefore, for every $u_j$, there are at most $2m$ vertices $v\in S_1(u_j)$ such that $\abs{S_1(v)\cap W_2}\ge c\kappa$ vertices from $S_1(v)$ in $W_2$. Then we can take a set of $m^4$ vertices $v$ in $S_1(u_j)$ along with $c\kappa$ vertices in every $S_1(v)$ avoiding $W_2$, forming the desired $(m^4,c\kappa)$-hub.
\end{proof}

\subsection{Constructing an adjuster}\label{subsec: robust adjuster}

In this section, we start by finding a simple adjuster in an expander despite the removal of any medium-sized vertex set by Lemma \ref{lem: robust simple adjuster}. Then, we link simple adjusters together to obtain a desired adjuster. First we state Lemma \ref{lem: robust simple adjuster}, a key ingredient of our proof. It finds a simple adjuster robustly in an expander $G$, that is, given any subset $W\subseteq V(G)$ with a moderate size, we can construct an adjuster in $G-W$.

\begin{lemma}\label{lem: robust simple adjuster}
    There exists some $\ep_1>0$ such that, for every $0<\ep_2<1$ and integer $s\ge 240$, there exists $d_0$ and $K$ such that the following holds for each $n\ge d\ge d_0$ and $d\ge \log^sn$. \\
    {\rm(\romannumeral1)} If $G$ is a $\mathsf{T}K_{\sqrt{d}}^{(2)}$-free n-vertex $(\ep_{1},\ep_{2}d)$-expander with $\delta(G)\ge d$ and $n\ge Kd$, then for any vertex set $W\subseteq G$ satisfing $\abs{W}\le 10D$ with $D=dm^4/10^7$, we have that $G-W$ contains a $(D,m/4,1)$-adjuster. \\
    {\rm(\romannumeral2)}  If $G$ is a $C_4$-free n-vertex $(\ep_{1},\ep_{2}d^2)$-expander with $\delta(G)\ge d$ and $n\ge Kd^2$, then for any vertex set $W\subseteq G$ satisfing $\abs{W}\le 10D$ with $D=d^2m^4/10^7$, we have that $G-W$ contains a $(D,m/4,1)$-adjuster.
\end{lemma}

Using Lemma \ref{lem: robust simple adjuster}, we can find many vertex-disjoint simple adjusters. Then we can connect them together into a large adjuster.

\begin{proof}[Proof of Lemma~\ref{lem: robust adjuster}.]
	The proof ideas of (\romannumeral1) and (\romannumeral2) are similar, and we choose $\kappa$ as in \eqref{kappa}. Thus $D=\kappa^2m^4/10^7$. Let $\ep_1>0$, and $K$, $d_0$ be sufficiently large. We prove the lemma by induction on $r$. For $r=1$, as $\abs{W}\le D/\log^3\frac{n}{\kappa^2}$, by Lemma \ref{lem: robust simple adjuster}, $G-W$ contains a $(D,m/4,1)$-adjuster, which is also a $(D,m,1)$-adjuster.

    Suppose then, for some $r$ with $1\le r< \frac{1}{10}\kappa^2m^2$,  $G-W$ contains a $(D,m,r)$-adjuster, denoted by $\mathcal{A}_1=(v_1,F_1,v_2,F_2,A_1)$. Let $W_1=W\cup A_1 \cup V(F_1)\cup V(F_2)$, we have $\abs{W_1}\le 4D$. Then Lemma \ref{lem: robust simple adjuster} shows that $G-W_1$ contains a $(D,m/4,1)$-adjuster $\mathcal{A}_2=(v_3,F_3,v_4,F_4,A_2)$. Note that $\abs{F_1\cup F_2}=\abs{F_3\cup F_4}=2D=: x$, and $\abs{W\cup A_1\cup A_2}\le D/\log^3\frac{n}{\kappa^2}+20rm\le D/\log^3\frac{n}{\kappa^2}+2\kappa^2m^3\le 2D/\log^3\frac{n}{\kappa^2}$. By \eqref{unit5}, we have $1/4\cdot\ep(x) x\ge 1/4\cdot\ep(n) x\ge 2D/\log^3\frac{n}{\kappa^2}\ge \abs{W\cup A_1\cup A_2}$, then by Lemma \ref{lem: small-diameter-lemma}, there is a shortest path $P$ of length at most $m$, from $V(F_1)\cup V(F_2)$ to $V(F_3)\cup V(F_4)$ avoiding $W\cup A_1\cup A_2$.

    We can assume, without loss of generality, that $P$ is a path from $V(F_1)$ to $V(F_3)$. Then by extending $P$, using that $F_1$ and $F_3$ are $(D,m)$-expansions of $v_1$ and $v_3$ respectively, we can get a $v_1,v_3$-path $Q\subseteq F_1\cup P\cup F_3$ of length at most $3m$. Now we claim that $(v_2,F_2,v_4,F_4,A_1\cup A_2\cup V(Q))$ is a $(D,m,r+1)$-adjuster. Indeed, we easily have that $\mathbf{A1}$ and $\mathbf{A2}$ hold, and $\mathbf{A3}$ holds as $\abs{A_1\cup A_2\cup V(Q)}\le 10mr+10\cdot(m/4)+3m\le 10(r+1)m$.
    Finally, let $l=l(\mathcal{A}_1)+l(\mathcal{A}_2)+l(Q)$. For every $i\in \{0,1,\cdots,r+1\}$, there are some $i_1\in \{0,1,\cdots,r\}$ and $i_2\in \{0,1\}$ such that $i=i_1+i_2$. Let $P_1$ be a $v_1,v_2$-path in $G[A_1\cup\{v_1,v_2\}]$ of length $l(\mathcal{A}_1)+2i_1$ and let $P_2$ be a $v_3,v_4$-path in $G[A_2\cup\{v_3,v_4\}]$ of length $l(\mathcal{A}_2)+2i_2$. Thus, $P_1\cup Q\cup P_2$ is a $v_2,v_4$-path in $G[A_1\cup A_2\cup V(Q)]$ of length $l+2i$, and therefore $\mathbf{A4}$ holds.
\end{proof}

Now we start to prove Lemma \ref{lem: robust simple adjuster}. We first introduce the concept of octopus used to find a simple adjuster.

\begin{defn}[Octopus]\label{octopus}
	Given integers $r_1,r_2,r_3,r_4>0$, an \emph{$(r_1,r_2,r_3,r_4)$-octopus} $\mathcal{B}=(A,R,\\ \mathcal{D},\mathcal{P})$ is a graph consisting of a \emph{core} $(r_1,r_2,1)$-adjuster $A$, one of the ends of $A$, called $R$, and
	\begin{itemize}
	\item a family $\mathcal{D}$ of $r_3$ vertex-disjoint $(r_1,r_2,1)$-adjusters, which are disjoint from $A$, and
	\item a minimal family $\mathcal{P}$ of internally vertex-disjoint paths of length at most $r_4$, such that each adjuster in $\mathcal{D}$ has at least one end which is connected to $R$ by a subpath from a path in $\mathcal{P}$, and all the paths are disjoint from all center sets of the adjusters in $\mathcal{D}\cup \{A\}$. It is easy to observe that $\abs{\mathcal{P}}\le \abs{\mathcal{D}}$.
    \end{itemize}
\end{defn}

\begin{figure}[htbp]\label{L2}
	\centering
	\includegraphics[scale=0.36]{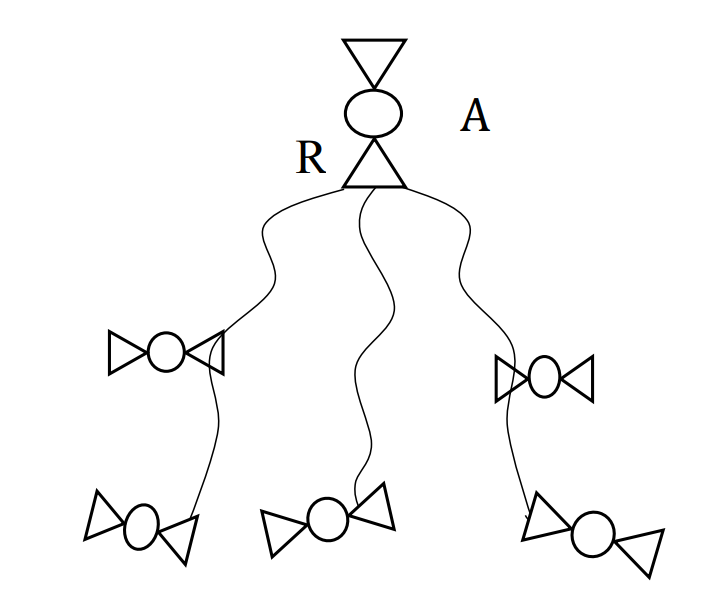}
	\caption{octopus}
\end{figure}

An outline of the proof of Lemma \ref{lem: robust simple adjuster} is as follows. We can find an expander $H\subseteq G-W$, but note that $H$ may be much smaller than $G$, so we may not find a simple adjuster of desired size directly. To overcome this, we first find many vertex-disjoint small simple adjusters. Then we connect them such that by averaging one end of a small simple adjuster would be linked to many others by internally vertex-disjoint paths, forming an octopus (see Claim \ref{claim: connect candy}). This process can be seen as that we expand one end of the small simple adjuster to a desired size. Finally, using the similar idea, we can expand another end to get the desired simple adjuster.

\begin{proof}[Proof of Lemma~\ref{lem: robust simple adjuster}.]
	The proof ideas of (\romannumeral1) and (\romannumeral2) are similar, and we choose $\kappa$ as in \eqref{kappa}. Recall $m$ is the smallest even integer that larger than $80\log^{4}\frac{n}{\kappa^2}$.
    Let $K$ be sufficiently large.  We first prove that there are $m^{30}$ pairwise disjoint $(\kappa^2/800,m/400,1)$-adjusters in $G-W$. Note that such $m^{30}$ adjusters have total size at most $(2\cdot \kappa^2/800+10\cdot m/400)\cdot m^{30}\le \kappa^2m^{30}/20$. It suffices to show that for any set $W'\subseteq V(G)$ with size $\kappa^2m^{30}/10$, there is a $(\kappa^2/800,m/400,1)$-adjuster in $G-W'$.

    For (\romannumeral1), by Lemma \ref{lem: robust degree} with $(W,x)=(W',30)$, we have $d(G-W')\ge d/2$, and by Corollary \ref{cor: bipartite expander} with $G=G-W'$, there exists a bipartite $(\ep_1,\ep_2d)$-expander $H\subseteq G-W'$ with $\delta (H)\ge d/16$. Then by Lemma \ref{lem: small-diameter-lemma}, there exists a shortest cycle $C$ in $H$ of length at most $m/16$ and whose length is denoted by $2r$. Thereafter, we arbitrarily pick two vertices $v_1,v_2\in V(C)$ of distance $r-1$ apart on $C$, together with $d/800$ distinct vertices in $N_{H-C}(v_1),N_{H-C}(v_2)$ respectively, then we get a $(d/800,m/400,1)$-adjuster as desired. Note that every vertex in the ends of the adjuster has no neighbour in the same end, except for $v_1$ and $v_2$.

    For (\romannumeral2), by Corollary \ref{cor: c4-free robust degree} with $(W,x)=(W',30)$, we have $d(G-W')\ge d/2$, and by Corollary \ref{cor: bipartite expander} with $G=G-W'$, there exists a bipartite $(\ep_1,\ep_2d^2)$-expander $H\subseteq G-W'$ with $\delta (H)\ge d/16$. Then by Lemma \ref{lem: small-diameter-lemma}, there exists a shortest cycle $C$ in $H$ of length at most $m/16$ and whose length is denoted by $2r$. We can arbitrarily pick one vertex $v_1\in V(C)$, which has at least $d/20$ neighbours in $H-C$. As $G$ is $C_4$-free, every pair of vertices in $N_{H-C}(v_1)$ has no other common neighbour than $v_1$. Thus $N_{H-C}(N_{H-C}(v_1))$ has size at least $(d/16-2r)d/20\ge d^2/400$. Choose another vertex $v_2\in V(C)$ which is at distance $r-1$ from $v_1$ in $C$. Now we find two vertex-disjoint $(d^2/800,2)$-expansions of $v_1$ and $v_2$ respectively, $F_1$, $F_2$ say, such that every vertex in $N_{F_1}(v_1)$ and $N_{F_2}(v_2)$ has at most $d$ neighbours in the same expansion. Thus $F_1$ and $F_2$ together with $C$ form the $(d^2/800,m/400,1)$-adjuster as desired.

    We will sometimes call an adjuster is \emph{touched} by a path if they intersect in at least one vertex, and \emph{untouched} otherwise. Now we give a claim, which helps us find internally vertex-disjoint short paths connecting a vertex set to many ends of different adjusters.

    \begin{claim}\label{claim: connect candy}
    Given $G$, $m$, $s$ and $\kappa$ as above. Let $x\ge5$ and $X\subseteq V(G)$ be an arbitrary vertex set of size at most $\kappa^2m^{x-1}/2$. Let $B\subseteq G-X$ be a graph with order at least $\kappa^2m^{x}/800$ and $\mathcal{U}$ be a subfamily of vertex-disjoint $(\kappa^2/800,m/400,1)$-adjusters in $G-(X\cup V(B))$ with $\abs{\mathcal{U}}\ge m^{2x}$. Let $\mathcal{P}_B$ be a maximum collection of internally vertex-disjoint paths of length at most $m/8$ in $G-X$, each connecting $V(B)$ to one end from different adjusters in $\mathcal{U}$. Then $V(B)$ can be connected to $1600m^{x+4}$ ends from different adjusters in $\mathcal{U}$ via a subpath from a path in $\mathcal{P}_B$.
    \end{claim}

    \begin{proof}
    Suppose to the contrary that there are less than $1600m^{x+4}$ such ends from different adjusters, and denote by $P$ the set of internal vertices of those paths. Then we have $\abs{P}\le 1600m^{x+4}\cdot m/8=200m^{x+5}$, and there are at least $m^{2x}-1600m^{x+4}$ adjusters in $\mathcal{U}$ untouched by the paths in $\mathcal{P}_B$. Arbitrarily pick $m^x$ adjusters among those adjusters, and let $B'$ be the union of their ends. Then we have $\abs{B'}=m^x\cdot 2\kappa^2/800=\kappa^2m^x/400$.
    As $s\ge 240$, and thus $d\ge m^{60}$, we have
    $$\abs{X\cup P}\le \frac{\kappa^2m^{x-1}}{2}+200m^{x+5}\le \kappa^2m^{x-1}.$$
    Let $y:=\kappa^2m^{x}/800$. By \eqref{unit5}, we have $1/4\cdot\ep(y) y\ge 1/4\cdot\ep(n) y\ge \kappa^2m^{x-1}\ge \abs{X\cup P}$. Then by Lemma \ref{lem: small-diameter-lemma}, there is a path of length at most $m/8$ from $V(B)$ to $V(B')$, avoiding $X\cup P$, that is, this path connects $V(B)$ to an end of one more adjuster in $\mathcal{U}$, contrary to the maximality of $\mathcal{P}_B$.
    \end{proof}

    In order to build a $(D,m/4,1)$-adjuster, we shall first construct many octopuses using those $(\kappa^2/800,m/400,1)$-adjusters we found above. Let $Z$ be the union of the center sets and core vertices of all those adjusters.
    \begin{claim}
    There are $m^5$ $(\kappa^2/800,m/400,800m^4,m/8)$-octopuses $\mathcal{B}_j=(A_j,R_j,\mathcal{D}_j,\mathcal{P}_j)$, $1\le j\le m^5$ in $G-W$ such that the following rules hold.
    \begin{itemize}
        \item[$\mathbf{D1}$] $A_j$ are pairwise disjoint adjusters, $1\le j\le m^5$.
        \item[$\mathbf{D2}$] $A_i\notin \mathcal{D}_j$, $1\le i,j\le m^5$.
        \item[$\mathbf{D3}$] $\mathcal{D}_j$ contains every adjuster other than $A_j$ which intersects at least one path in $\mathcal{P}_j$, $1\le j\le m^5$.
	    \item[$\mathbf{D4}$] Paths in $\mathcal{P}_i$ are vertex-disjoint from $Z$ and $A_j$, except for $i=j$, $1\le i,j\le m^5$.
    	\item[$\mathbf{D5}$] Every two paths from distinct $\mathcal{P}_i$, $\mathcal{P}_j$ are mutually vertex-disjoint, $1\le i<j\le m^5$.
    \end{itemize}
    \end{claim}

    \begin{proof}
    We will construct the octopuses iteratively. Suppose that we have constructed $q$ octopuses so far. Let $W_1=W\cup Z$. Let $U$ be the union of the vertex sets of the ends of the core adjusters of octopuses we have constructed. Then we have $\abs{U}\le m^5\cdot 2\kappa^2/800 =\kappa^2m^5/400$.
    Let us call an adjuster \emph{used} if it is used to construct an octopus, and \emph{unused} otherwise. There are at most $m^5\cdot (800m^4+1)$ adjusters are used until now, so there are more than $m^{20}$ unused adjusters.

    Let $P=\cup_{j=1}^{q}V(\mathcal{P}_j)$. Then $\abs{P}\le m/8\cdot 800m^4\cdot m^5\le m^{11}$. Arbitrarily pick a subfamily $\mathcal{B}$ of $m^6$ unused adjusters, and let $B$ be the union of their ends. Then $\abs{B}=m^6\cdot 2\kappa^2/800=\kappa^2m^6/400$. Let $\mathcal{U}$ be the family of unused adjusters, except for the adjusters we picked above. Then we have $\abs{\mathcal{U}}\ge m^{12}$. Note that $\abs{W_1\cup U\cup P}\le (10D+ m^{30}\cdot m/16)+ \kappa^2m^5/400+m^{11}\le \kappa^2m^5/2$ as $d\ge\log^sn\ge m^{60}$. Then by Claim \ref{claim: connect candy} with $(B,\mathcal{U},x,X)=(B,\mathcal{U},6,W_1\cup U\cup P)$, $V(B)$ can be connected to $1600m^{10}$ ends from different adjusters in $\mathcal{U}$ via some internally vertex-disjoint paths of length at most $m/8$ in $G-W_1-U-P$.

    By the pigeonhole principle, there exists an adjuster in $\mathcal{B}$, say $A_{q+1}$, such that $A_{q+1}$ has an end $R_{q+1}$ connected to at least $800m^4$ adjusters, say $\mathcal{D}'_{q+1}$, via a subfamily of internally vertex-disjoint paths, denoted by $\mathcal{P}'_{q+1}$, and then observe that we can choose $\mathcal{D}_{q+1}\subseteq \mathcal{D}'_{q+1}$ of exactly $800m^4$ adjusters and a system of paths $\mathcal{P}_{q+1}$ such that every path is an initial segment of one in $\mathcal{P}'_{q+1}$ and thus $\mathbf{D3}$ holds.
    Denote by $L_{q+1}$ the other end of $A_{p+1}$. As we only use adjusters outside the octopuses which have been constructed during the process, $\mathbf{D1}$ and $\mathbf{D2}$ hold. As we find paths in $\mathcal{P}_{q+1}$ avoiding $W_1\cup U\cup P$, $\mathbf{D4}$ and $\mathbf{D5}$ hold. Thus, $A_{q+1}$,$R_{q+1}$, $\mathcal{D}_{q+1}$ and $\mathcal{P}_{q+1}$ form a $(\kappa^2/800,m/400,800m^4,m/8)$-octopus.
    \end{proof}

    Now we have $m^5$ octopuses $\mathcal{B}_j=(A_j,R_j,\mathcal{D}_j,\mathcal{P}_j)$, $1\le j\le m^5$. Let $B_1$ be the union of $L_j$, $1\le j\le m^5$. Then we have $\abs{B_1}=m^5\cdot \kappa^2/800$. There are at most $m^5\cdot (800m^4+1)$ adjusters used, and thus at least $m^{20}$ adjusters unused. Let $\mathcal{U}'$ be the family of unused adjusters. Reset $P=\bigcup_{j=1}^{m^5}V(\mathcal{P}_j)$, then $\abs{P}\le m^{11}$.
    By definition, inside each $\mathcal{B}_j=(A_j,R_j,\mathcal{D}_j,\mathcal{P}_j)$, $ j\in [m^5]$, every adjuster $A \in \mathcal{D}_j$ intersects $V(\mathcal{P}_j)$ and thus there exists a shortest path in $A$ of length at most $m/400$ connecting a core vertex of $A$ to $V(\mathcal{P}_j)$, and denote by $\mathcal{Q}_j$ the disjoint union of such paths taken over all adjusters in $\mathcal{D}_j$. Let $Q=\cup_{j=1}^{m^5}V(\mathcal{Q}_j)$. Then $\abs{Q}\le m^5\cdot 800m^4\cdot (m/400+1)\le m^{11}$.
    Recall $s\ge 240$ and $d\ge\log^sn\ge m^{60}$, we have $\abs{W_1\cup P\cup Q}\le (10D+ m^{30}\cdot m/16)+m^{11}+m^{11}\le \kappa^2m^4/2$.
    Then by Claim \ref{claim: connect candy} with $(B,\mathcal{U},x,X)=(B_1,\mathcal{U}',5,W_1\cup P\cup Q)$, $V(B_1)$ can be connected to $800m^{9}$ ends from different adjusters in $\mathcal{U}'$ via some internally vertex-disjoint paths of length at most $m/8$ in $G-W_1-P-Q$.

    By the pigeonhole principle, there exists a core adjuster $A_k$ such that $L_{k}$ is connected to a family $\mathcal{D}'_{k}$ of at least $800m^4$ adjusters, via a subfamily of internally vertex-disjoint paths, denote by $\mathcal{P}'_{k}$. Then $A_{k}$, $L_{k}$ $\mathcal{D}'_{k}$ and $\mathcal{P}'_{k}$ form a $(\kappa^2/800,m/400,800m^4,m/8)$-octopus. Note that $(A_{k}, R_{k}, \mathcal{D}_{k}, \mathcal{P}_{k})$ is also a  $(\kappa^2/800,m/400,800m^4,m/8)$-octopus.

    For the adjuster $A_{k}$, denote by $C_{k}$ the center vertex set of $A_{k}$, and note that $L_{k}$, $R_{k}$ are $(\kappa^2/800,m/400)$-expansions of vertices $v_1$, $v_2$ respectively. Let $F_1':=G[V(L_{k})\cup V(\mathcal{P}'_{k})\cup V(\mathcal{D}'_{k})]$, and $F_2'$ be the component of $G[V(R_{k})\cup V(\mathcal{P}_{k})\cup V(\mathcal{D}_{k})]-V(\mathcal{P}'_{k})$ containing $v_2$. Indeed, paths in $\mathcal{P}_{k}$ and $\mathcal{P}'_{k}$ are disjoint from $Z$, and $V(\mathcal{P}_{k})$ and $V(\mathcal{P}'_{k})$ are disjoint. Recall that for every adjuster in $\mathcal{D}_{k}$, every vertex in the ends of the adjuster has at most $\kappa$ neighbours in the same end, except for its core vertices. As $d\ge\log^sn\ge m^{60}$, and $V(\mathcal{P}'_{k})$ is disjoint from $Z$ and $Q$, $F_2'$ has size at least
    $$\abs{V(\mathcal{D}_{k})}-\kappa \abs{V(\mathcal{P}'_{k})} \ge800m^4\cdot \frac{\kappa^2}{800}\cdot 2-\kappa \cdot\frac{m}{8} \cdot 800m^4\ge \kappa^2m^4,$$
    and the distance between $v_2$ and each $v\in V(F_2')$ is at most $m/400+m/8+m/400+m/32+m/400\le m/4$. Then by Proposition \ref{prop: expansion}, there exists a subgraph of $F_2'$, denoted by $F_2$, which is a $(\kappa^2m^4,m/4)$-expansion of $v_2$. Similarly, we can find $F_1$, which is a $(\kappa^2m^4,m/4)$-expansion of $v_1$. Recall that $C_{k}\cup \{v_1,v_2\}$ is an even cycle of length $2r'\le m/16$, and the distance between $v_1$ and $v_2$ on $C_{k}\cup \{v_1,v_2\}$ is $r'-1$. Thus, $(v_1,F_1,v_2,F_2,C_{k})$ is a $(\kappa^2m^4,m/4,1)$-adjuster, and by Proposition \ref{prop: expansion}, there exists a $(D,m/4,1)$-adjuster in $G-W$.
\end{proof}

\subsection{Connecting vertices with paths of desired length}\label{subsec: connect four}

Our goal now is connecting two vertex sets to two expansions with two paths, whose combined length is in the desired range. We first connect one vertex set to an expansion by the following lemma.

\begin{lemma}\label{lem: connect two}
    For any $0<\ep_1,\ep_2<1$, there exists $d_0$ and $K$ such that the following holds for $n\ge d\ge d_0$ and $d\ge \log^sn$. \\
    {\rm(\romannumeral1)} Let $G$ be a $\mathsf{T}K_{\sqrt{d}}^{(2)}$-free n-vertex $(\ep_{1},\ep_{2}d)$-expander with $\delta(G)\ge d$ and $n\ge Kd$. Let $D=dm^4/10^7$. Suppose $U\subseteq V(G)$ is a vertex set with $\abs{U}\ge D$, $F$ is a $(D,m)$-expansion of $v$ in $G-U$, and $W\subseteq V(G)\backslash (U\cup V(F))$ satisfies $\abs{W}\le 2D/\log^3\frac{n}{d}$. Then for any $l\le dm^3$, there is a $v,v'$-path in $G-W$ for some $v'\in U$ of length between $l$ and $l+11m$. \\
	{\rm(\romannumeral2)} Let $G$ be a $C_4$-free n-vertex $(\ep_{1},\ep_{2}d^2)$-expander with $\delta(G)\ge d$ and $n\ge Kd^2$. Let $D=d^2m^4/10^7$. Suppose $U\subseteq V(G)$ is a vertex set with $\abs{U}\ge D$, $F$ is a $(D,m)$-expansion of $v$ in $G-U$, and $W\subseteq V(G)\backslash (U\cup V(F))$ satisfies $\abs{W}\le 2D/\log^3\frac{n}{d^2}$. Then for any $l\le d^2m^3$, there is a $v,v'$-path in $G-W$ for some $v'\in U$ of length between $l$ and $l+11m$.
\end{lemma}

\begin{proof}
	The proof ideas of (\romannumeral1) and (\romannumeral2) are similar, and we choose $\kappa$ as in \eqref{kappa}. Thus $D=\kappa^2m^4/10^7$ and $l\le \kappa^2m^3$.
    Let $K$ be sufficiently large. Let $(P,v_1,F_1)$ be such that $l(P)$ is maximised subject to the following properties.
    \begin{itemize}
	\item[$\mathbf{E1}$] $P$ is a $v,v_1$-path in $G-W$.
	\item[$\mathbf{E2}$] $l(P)\le l+7m$.
	\item[$\mathbf{E3}$] $F_1$ is a $(D,3m)$-expansion of $v_1$ in $G-W$ with $V(F_1)\cap V(P)=\{v_1\}$.
    \end{itemize}
    Note that $P=G[{v_1}]$, $v_1=v$, and $F_1=F$ satisfy $\mathbf{E1}$-$\mathbf{E3}$, therefore such a tuple $(P,v_1,F_1)$ exists.

    We claim that $l(P)\ge l$. Suppose to the contrary that $l(P) < l$. Note that for sufficiently large $K$, we have  $\abs{W\cup V(P\cup F_1)}\le 2D/\log^3\frac{n}{\kappa^2}+D+l\le 2D$. By Lemma \ref{lem: robust unit}, $G-W-V(P\cup F_1)$ contains a $(c\kappa,m^4,c\kappa,2m)$-unit $U'$ with core vertex $v_2$. Note that $\abs{W\cup V(P)}\le 3D/\log^3\frac{n}{\kappa^2}$. By \eqref{unit5}, we have $1/4\cdot\ep(D)D\ge 1/4\cdot\ep(n)D\ge 3D/\log^3\frac{n}{\kappa^2}\ge \abs{W\cup V(P)}$. Then by Lemma \ref{lem: small-diameter-lemma}, there is a path $Q'$ from $V(U')$ to $V(F_1)$ of length at most $m$, avoiding $W\cup V(P)\backslash \{v_1\}$. Without loss of generality we can assume that $Q'$ has endvertices $v'_1\in F_1$ and $v'_2\in U'$. By $\mathbf{E3}$ and the fact that $U'$ is a $(c\kappa,m^4,c\kappa,2m)$-unit, we can extend $Q'$ to a $v_1,v_2$-path $Q$ of length at most $2m+2+m+3m\le 7m$ which is vertex-disjoint from $P-v_1$.
    Using the definition of unit, we can get $F_2\subseteq U'\backslash V(Q)\cup \{v_2\}$ which is a $(D,3m)$-expansion of $v_2$. Let $P_1=P\cup Q$. As $l(Q)\le 7m$, $P_1$ is a $v,v_2$-path of length at least $l(P)+1$ and at most $l(P)+7m$. Then, $(P_1, v_2, F_2)$ satisfies $\mathbf{E1}$-$\mathbf{E3}$ with $l(P_1)>l(P)$, a contradiction. Therefore, $l(P)\ge l$.

    Now, note that $\abs{W\cup V(P_1)}\le 3D/\log^3\frac{n}{\kappa^2}$. By \eqref{unit5}, we have $\ep(D)D/4\ge \ep(n)D/4\ge 3D/\log^3\frac{n}{\kappa^2}\ge \abs{W\cup V(P_1)}$. Then by Lemma \ref{lem: small-diameter-lemma}, there is a path $R$ of length at most $m$, from some $r_1\in U$ to some $r_2\in V(F_1)$ avoiding $W\cup V(P)\backslash \{v_1\}$. Let $Q_1$ be a path from $v_1$ to $r_2$ in $F_1$ of length at most $3m$. Then, $P\cup Q_1\cup R$ is a $v,r_1$-path in $G-W$ of length at least $l(P)\ge l$ and at most $l+7m+3m+m=l+11m$ by $\mathbf{E2}$.
\end{proof}

Combining Lemma \ref{lem: connect two} with Lemma \ref{lem: small-diameter-lemma}, we can prove Lemma~\ref{cor: connect four} as follows.

\begin{proof}[Proof of Lemma~\ref{cor: connect four}.]
	The proof ideas of (\romannumeral1) and (\romannumeral2) are similar, and we choose $\kappa$ as in \eqref{kappa}. Thus $D=\kappa^2m^4/10^7$ and $l\le \kappa^2m^3$.
    Given $G$, $W$, $U_1$ and $U_2$ as stated, we choose $K$ to be sufficiently large. Note that $\abs{U_1\cup U_2}\ge 2D$ and $\abs{V(F_3)\cup V(F_4)}=2D$. By \eqref{unit5}, we have $1/4\cdot\ep(2D)\cdot 2D\ge 1/4\cdot\ep(n)\cdot 2D\ge D/\log^3\frac{n}{\kappa^2}\ge \abs{W}$. Then by Lemma \ref{lem: small-diameter-lemma}, there is a shortest path $P'\subseteq G-W$ from $U_1\cup U_2$ to $V(F_3)\cup V(F_4)$ of length at most $m$. Without loss of generality, we can assume that $P'$ goes from some $v_1\in U_1$ to $v'_3\in V(F_3)$. As $F_3$ is a $(D,m)$-expansion of $v_3$, there is a $v_1,v_3$-path $P$ of length at most $2m$ in $U_1\cup P'\cup F_3$.

    Let $W'=W\cup V(P)$. Then $\abs{W'}\le D/\log^3\frac{n}{\kappa^2}+2m+2\le 2D/\log^3\frac{n}{\kappa^2}$. By Lemma \ref{lem: connect two} with $(U,F,D,m,W,l)=(U_2,F_4,D,m,W',l)$, there is a path $Q$ in $G-W'$ from some $v_2\in U_2$ to $v_4$ of length between $l$ and $l+11m$. As $l\le l(P)+l(Q)\le l+13m$, the paths $P$ and $Q$ are desired.
\end{proof}

\section*{Acknowledgement}
We want to mention that Fern{\'a}ndez, Hyde, Liu, Pikhurko and Wu \cite{FHLPW} obtained a similar result regarding Theorem~\ref{thm: balanced subdivision}. The results in two papers are finished independently.

\begin{appendices}

\section{Dependent random choice}

The following lemma can be regarded as a bipartite version of dependent random choice, of which the proof follows from that of Fox and Sudakov \cite{F-S}.

\begin{lemma}\label{depend}
Given integers $a, t, n_1,n_2, c, r$ and a constant $\alpha>0$, let $G=(V_1,V_2,E)$ be a bipartite graph such that $|V_1|=n_1$, $|V_2|=n_2$ and $\abs{E}\ge \alpha n_1n_2$. If it holds that \[\alpha^tn_1-\binom{n_1}{r}\left(\frac{c}{n_2}\right)^t\ge a,\] then there exists a set $A_0\subseteq V_1$ of size at least $a$ such that every $r$-subset of $A_0$ has at least $c$ common neighbours in $V_2$.
\end{lemma}
\begin{proof}
Pick a set of $t$ vertices of $V_2$ uniformly at random with repetition, say $b_1,b_2,\cdots,b_t$. Let
$A$ denote the set of common neighbours for all vertices $b_i$ and $X=|A|$. By linearity of expectation, $$\mathbb{E}(X)=\sum_{v\in V_1}\mathbb{P}(b_1,\cdots,b_t\in N(v))=\sum_{v\in V_1}\left(\frac{d(v)}{n_2}\right)^t.$$
By the convexity of the function $f(x)=x^t$, we have
$$\mathbb{E}(X)\ge\frac{n_1(\sum_{v\in V_1}d(v)/n_1)^t}{n_2^t}\ge\frac{n_1(\alpha n_2)^t}{n_2^t}\ge\alpha^tn_1. $$
Let $Y$ denote the random variable counting the number of $r$-subsets in $A$ with fewer than $c$ common neighbours in $V_2$. Therefore, the probability that a randomly
chosen $b_i$ is one of the common neighbours of such an $r$-set is at most $\frac{c}{n_2}$. Hence, since we made random choices of $b_i$ uniformly and independently, the
probability that such an $r$-tuple be contained in $A$ is at most $\left(\frac{c}{n_2}\right)^t$. As there are at most $\binom{|V_1|}{r}$ subsets of size $r$, it follows that $$\mathbb{E}(Y)\le \binom{|V_1|}{r}\left(\frac{c}{n_2}\right)^t.$$
Again, by linearity of expectation, it holds that $$\mathbb{E}(X-Y)\ge\alpha^tn_1-\binom{n_1}{r}\left(\frac{c}{n_2}\right)^t\ge a.$$
Hence there exists a choice of $A$ for which $X-Y\ge a$. Delete
one vertex from each subset $r$-subset of $A$ with fewer than $c$ common neighbours and let $A_0$ be the remaining subset of $A$. Thus, $|A_0|\ge a$ and every $r$-subset of $A_0$ has at least $c$ common neighbours.
\end{proof}

\end{appendices}


\begin{thebibliography}{99}
	\setlength{\parskip}{0pt}
	\setlength{\itemsep}{0pt plus 0.3ex}
	\footnotesize
	
\bibitem{AKS}
N.~Alon, M.~Krivelevich, B.~Sudakov,
\newblock Tur{\'a}n numbers of bipartite graphs and related Ramsey-type questions.
\newblock \emph{Combinatorics, Probability and Computing}, 12, (2003), 477--494.


\bibitem{BLS}
J.~Balogh, H.~Liu, M.~Sharifzadeh,
\newblock Subdivisions of a large clique in $C_6$-free graphs.
\newblock \emph{Journal of Combinatorial Theory, Series B}, 112, (2015), 18--35.


\bibitem{B-Th}
B.~Bollob{\'a}s, A.~Thomason,
\newblock Proof of a conjecture of {M}ader, {E}rd{\H o}s and {H}ajnal on
topological complete subgraphs.
\newblock \emph{European Journal of Combinatorics}, 19, (1998), 883--887.

\bibitem{E}
P.~Erd\H{o}s,
\newblock Problems and results in graph theory and combinatorial analysis.
\newblock \emph{Graph theory and related topics (Proc. Conf. Waterloo, 1977)}, Academic Press, New York (1979), 153--163.

\bibitem{E-H}
P.~Erd\H{o}s, A.~Hajnal,
\newblock On topological complete subgraphs of certain graphs.
\newblock \emph{Annales Universitatis Scientiarum Budapestinensis de Rolando E{\"o}tv{\"o}s Nominatae Sectio Computatorica}, 7, (1969), 193--199.

\bibitem{FHLPW}
I.~Fern{\'a}ndez, J.~Hyde, H.~Liu, O.~Pikhurko, Z.~Wu,
\newblock Disjoint isomorphic balanced clique subdivisions.
\newblock arXiv preprint, arXiv:2204.12465.

\bibitem{F-S}
J.~Fox, B.~Sudakov,
\newblock Dependent random choice.
\newblock \emph{Random Structures and Algorithms}, 38(1-2), (2011), 68--99.

\bibitem{Jung}
H.A.~Jung,
\newblock Eine {V}erallgemeinerung des {$n$}-fachen {Z}usammenhangs f\"ur {G}raphen.
\newblock \emph{Mathematische Annalen}, 187, (1970), 95--103.

\bibitem{K-O-1}
D.~K{\"u}hn, D.~Osthus,
\newblock Topological minors in graphs of large girth.
\newblock \emph{Journal of Combinatorial Theory, Series B}, 86, (2002), 364--380.

\bibitem{K-O-2}
D.~K{\"u}hn, D.~Osthus,
\newblock Large topological cliques in graphs without a 4-cycle.
\newblock \emph{Combinatorics, Probability and Computing}, 13, (2004), 93--102.

\bibitem{K-O-3}
D.~K{\"u}hn, D.~Osthus,
\newblock Improved bounds for topological cliques in graphs of large girth.
\newblock \emph{SIAM Journal on Discrete Mathematics}, 20, (2006), 62--78.


\bibitem{KST}
T.~K\H{o}v{\'a}ri, V.T.~S{\'o}s, P.Tur{\'a}n,
\newblock On a problem of K.Zarankiewicz.
\newblock \emph{Colloquium Mathematicum}, 3, (1954), 50--57.
	
	
\bibitem{K-Sz-1}
J.~Koml{\'o}s, E.~Szemer{\'e}di,
\newblock Topological cliques in graphs.
\newblock \emph{Combinatorics, Probability and Computing}, 3, (1994), 247--256.

\bibitem{K-Sz-2}
J.~Koml{\'o}s, E.~Szemer{\'e}di,
\newblock Topological cliques in graphs {II}.
\newblock \emph{Combinatorics, Probability and Computing}, 5, (1996), 79--90.


\bibitem{Kuratowski}
J.~Kuratowski,
\newblock Sur le probleme des courbes gauches en topologie.
\newblock \emph{Fundamenta Mathematicae}, 16, (1930), 271--283.


\bibitem{LM1}
H. Liu, R.H. Montgomery,
\newblock A proof of Mader's conjecture on large clique subdivisions in $C_4$-free graphs. \newblock \emph{Journal of the London Mathematical Society}, 95(1), (2017), 203--222.

\bibitem{LM2}
H. Liu, R.H. Montgomery,
\newblock A solution to Erd\H os and Hajnal's odd cycle problem.
\newblock \emph{Journal of the American Mathematical Society}, https://doi.org/10.1090/jams/1018.


\bibitem{Mader-1}
W.~Mader,
\newblock Homomorphieeigenschaften und mittlere Kantendichte von Graphen.
\newblock \emph{Mathematische Annalen}, 174, (1967), 265--268.

\bibitem{Mader-2}
W.~Mader,
\newblock Hinreichende Bedingungen f\H{u}r die Existenz von Teilgraphen, die zu einem vollst\H{a}ndigen Graphen hom\H{o}omorph sind.
\newblock \emph{Mathematische Nachrichten}, 53(1-6), (1972), 145--150.

\bibitem{Mader-3}
W.~Mader,
\newblock An extremal problem for subdivisions of $K_5^-$.
\newblock \emph{Journal of Graph Theory}, 30, (1999), 261--276.


\bibitem{Th-1}
C.~Thomassen,
\newblock Subdivisions of graphs with large minimum degree.
\newblock \emph{Journal of Graph Theory}, 8(1), (1984), 23--28.

\bibitem{Th-2}
C.~Thomassen,
\newblock Problems 20 and 21.
\newblock In \emph{Graphs, Hypergraphs and Applications}, H.Sachs,Ed.:217. Teubner. Leipzig.,1985.

\bibitem{Th-3}
C.~Thomassen,
\newblock Configurations in graphs of large minimum degree, connectivity, or chromatic number.
\newblock \emph{Annals of the New York Academy of Sciences}, 1, (1989), 402--412.

\bibitem{V}
J.~Verstra{\"e}te,
\newblock A note on vertex-disjoint cycles.
\newblock \emph{Combinatorics Probability and Computing}, 1, (2002), 92--102.

\bibitem{WY}
Y.~Wang,
\newblock Balanced subdivisions of a large clique in graphs with high average degree.
\newblock arXiv preprint, arXiv:2107.06583v1.

\end{thebibliography}
\end{document}